\documentclass[12pt]{article}
\usepackage{epsfig,amsmath,amsthm,amssymb,graphics,amsfonts,psfrag}
\usepackage[letterpaper,margin=1in]{geometry}

\theoremstyle{plain}
\newtheorem{thm}{Theorem}[section]

\newtheorem{prop}[thm]{Proposition}

\theoremstyle{definition}
\newtheorem{defn}[thm]{Definition}
\newtheorem{ex}[thm]{Example}

\def\R{\mathbb{R}}

\def\I{\infty}

\newcommand{\be}{\begin{equation}}
\newcommand{\ee}{\end{equation}}
\newcommand{\bea}{\begin{eqnarray}}
\newcommand{\eea}{\end{eqnarray}}
\newcommand{\beann}{\begin{eqnarray*}}
\newcommand{\eeann}{\end{eqnarray*}}
\newcommand{\benn}{\begin{equation*}}
\newcommand{\eenn}{\end{equation*}}

\def\ra{\rightarrow}
\def\I{\infty}

\begin{document}

\date{}
\title{Characterizing Slow Exit and Entrance Points}
\author{Christian Kuehn\thanks{Center for Applied Mathematics, Cornell University, email: ck274@cornell.edu}}

\maketitle

\begin{abstract}
Geometric Singular Perturbation Theory (GSPT) and Conley Index Theory are two powerful techniques to analyze dynamical systems. Conley already realized that using his index is easier for singular perturbation problems. In this paper, we will revisit Conley's results and prove that the GSPT technique of Fenichel Normal Form can be used to simplify the application of Conley index techniques even further. We also hope that our results provide a better bridge between the different fields. Furthermore we show how to interpret Conley's conditions in terms of averaging. The result are illustrated by the two-dimensional van der Pol equation and by a three-dimensional Morris-Lecar model.  
\end{abstract}



\section{Introduction}

We start by outlining the context and results of the paper. Singular perturbation theory often involves a distingushed parameter, usually denoted $\epsilon$, which is assumed to be small and positive. For $\epsilon=0$, the dynamical system is ``degenerate'' but it might also be easier to analyze. Ordinary differential equations (ODEs) with two different time scales, so-called fast-slow systems, are a very important class of singular perturbation problems. In this context, the parameter $\epsilon$ describes the separation of time scales. One successful strategy to analyze fast-slow systems is to understand the case $\epsilon=0$ and using this knowledge to try to prove perturbation results for $\epsilon>0$ sufficiently small; many geometric and asymptotic methods follow this pattern.\\

Conley Index Theory can be applied to wide classes of dynamical systems which do not have to be singularly perturbed. The goal is to convert the problem into an algebraic/topological question. Conley already proved that this approach can be substanially simplified for singular perturbation problems. Mischaikow and co-workers developed Conley's index theory for fast-slow systems recently even further. The problem with the broad applicability of the theory is that it is still very technical. It has been demonstrated in low-dimensional examples that easily applicable geometric versions of the theory should be possible. One possible generalization to higher dimensions has been proposed but requires a rather complicated and lengthy topological construction.\\

The goal of the current paper is to start building a simpler bridge between GSPT and Conley Index Theory; in particular, we are going to show that a fundamental result due to Conley can be simplified using Fenichel Normal Form. The result we prove shows that if we isolate an invariant set by a suitable neighbourhood $N$, then $N$ is of the form required by Conley Index Theory if the slow motion for $\epsilon=0$ is transverse to the boundary of $N$. We also investigate the case when the fast motion of the system has periodic orbits and discuss further generalizations. We hope that the new results will be of interest from the viewpoint of GSPT as well as Conley Index Theory.\\

The paper is structured as follows. In Section 2 we describe the background from fast-slow systems theory and in Section 3 a similar exposition is given for Conley Index Theory with a focus on the application to fast-slow systems. Both introductions are a little more detailed than strictly necessary to accomodate the two different perspectives with respect to background knowledge. In Section 4 we prove the main result for equilibrium points of the fast motion. In Section 5 a result for the case of periodic orbits is given and a generalization to any bounded invariant set is outlined. Furthermore the role of averaging is explained. In Section 6 we demonstrate the applicability of the result by two examples.   

\section{Fast-Slow Systems}

Several viewpoints have influenced the development of multiple time scale or fast-slow systems starting with asymptotic analysis \cite{MisRoz,Eckhaus} using techniques like matched asymptotic expansions \cite{KevorkianCole,Lagerstrom}. A geometric theory focusing on invariant manifolds was developed \cite{Fenichel4,Jones,TikhonovVasilevaSveshnikov} which is now commonly known as Fenichel theory due to Fenichel's seminal work \cite{Fenichel4}. There was also significant influence by a group using nonstandard analysis \cite{DienerDiener,BenoitCallotDienerDiener}.\\ 

We shall focus on the geometric viewpoint in this paper. The term ``Geometric Singular Perturbation Theory'' (GSPT) is used to encompass Fenichel theory and further geometric methods developed over the last three decades in the context of multiple time scale problems. The general formulation of a \textit{fast-slow system} of ordinary differential equations (ODEs) is
\be
\label{eq:basic1}
\begin{array}{lcl}
\epsilon \dot{x}&=&\epsilon\frac{dx}{d\tau}=f(x,y,\epsilon),\\
\dot{y}&=&\frac{dy}{d\tau}=g(x,y,\epsilon),\\
\end{array}
\ee
where $(x,y)\in\R^m\times \R^n$ and $\epsilon$ is a small parameter $0<\epsilon\ll 1$ representing the ratio of time scales. The functions $f:\R^m\times\R^n \times \R\ra \R^m$ and $g:\R^m\times\R^n\times \R\ra \R^n$ will be assumed to be sufficiently smooth. The variables $x$ are fast and the variables $y$ are slow and we can change in \eqref{eq:basic1} from the slow time scale $\tau$ to the fast time scale $t=\tau/\epsilon$ which yields:
\be
\label{eq:basic2}
\begin{array}{lcl}
x'&=&\frac{dx}{dt}=f(x,y,\epsilon),\\
y'&=&\frac{dy}{dt}=\epsilon g(x,y,\epsilon).\\
\end{array}
\ee 
We will also denote the vector field \eqref{eq:basic2} by $z'=F(z)$ where $F=(f,\epsilon g)$ and $z=(x,y)$. The first major idea to analyze \eqref{eq:basic1}-\eqref{eq:basic2} is to consider the \textit{singular limit} as $\epsilon\ra 0$. 

\begin{defn}
Setting $\epsilon=0$ in \eqref{eq:basic2} gives
\be
\label{eq:basic_fss}
\begin{array}{lcl}
x'&=&f(x,y,0),\\
y'&=&0,\\
\end{array}
\ee 
which is system of ODEs parametrized by the slow variables $y$. We call \eqref{eq:basic_fss} the \textit{fast subsystem} or \textit{layer equations}. The associated flow is called the \textit{fast flow}.
\end{defn}

\begin{defn}
Considering the singular limit $\epsilon =0$ for \eqref{eq:basic1} yields:
\be
\label{eq:basic_sf}
\begin{array}{lcl}
0&=&f(x,y,0),\\
\dot{y}&=&g(x,y,0).\\
\end{array}
\ee
System \eqref{eq:basic_sf} is a differential-algebraic equation (DAE) called \textit{slow subsystem} or \textit{reduced system}. The associated flow is called the \textit{slow flow}.
\end{defn}

One goal of GSPT is to use the fast and slow subsystems to understand the dynamics of the full system \eqref{eq:basic1}-\eqref{eq:basic2} for $\epsilon>0$.

\begin{defn} 
The algebraic constraint of \eqref{eq:basic_sf} defines the \textit{critical manifold}
\benn
C:=\{(x,y)\in\R^m\times \R^n|f(x,y,0)=0\}.
\eenn 
\end{defn}

Note that it is possible that $C$ is not an actual manifold \cite{KruSzm4} but we shall not consider this case here. The points in $C$ are equilibrium points for the fast subsystem \eqref{eq:basic_fss}. 

\begin{ex}
\label{ex:fss}
Consider the following very simple planar fast-slow system
\be
\label{eq:ex_fss}
\begin{array}{lcl}
\epsilon\dot{x}&=&y-x^2,\\
\dot{y}&=&-1.\\
\end{array}
\ee
The critical manifold $C=\{(x,y)\in\R^2:y=x^2\}$ is a parabola. Observe that the slow flow on $C$ is $\dot{y}=-1$ so that under this flow any initial condition on $C$ will ``flow down'' to the origin $(x,y)=(0,0)$; see single arrows in Figure \ref{fig:fig1}(a). The fast subsystem is $x'=y-x^2$ which has one stable equilibrium and one unstable one for $y>0$, a saddle-node (or fold) bifurcation for $y=0$ and no equilibria for $y<0$; the flow is indicated by double arrows in Figure \ref{fig:fig1}(a).     
\end{ex}

\begin{figure}[htbp]
\psfrag{x}{$x$}
\psfrag{y}{$y$}
\psfrag{C}{$C$}
\psfrag{a}{(a)}
\psfrag{b}{(b)}
\psfrag{gamma}{$\gamma$}
\psfrag{Meps}{$M_\epsilon$}
	\centering
		\includegraphics[width=0.9\textwidth]{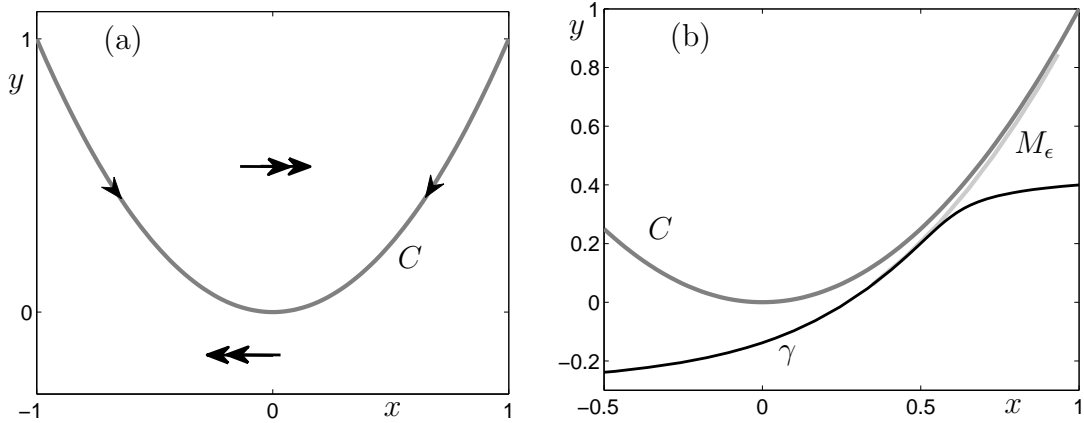}
	\caption{\label{fig:fig1}(a) Critical manifold $C$ (dark grey) of \eqref{eq:ex_fss}; the fast and slow flows are indicated by double and single arrows respectively. (b) $C$ (dark grey) and a slow manifold $M_\epsilon$ (light grey) obtained from Fenichel's Theorem \ref{thm:fenichel1} are shown together with a trajectory $\gamma$ of \eqref{eq:ex_fss} starting at $\gamma(0)=(1,0.4)$ at parameter value $\epsilon=0.05$. Observe that $\gamma$ quickly approaches $M_\epsilon$, then tracks it (actually $O(e^{-K/\epsilon})$-close) until the fast-slow structure breaks down near the fold point $(x,y)=(0,0)$.}	
\end{figure}

\begin{defn}
A subset $M\subset C$ is called \textit{normally hyperbolic} if the $m\times m$ matrix $(D_xf)(p)$ of first partial derivatives with respect to the fast variables has no eigenvalues with zero real part for all $p\in M$; this condition is equivalent to requiring that points $p\in S$ are hyperbolic equilibria of the fast subsystem \eqref{eq:basic_fss}. 
\end{defn}

We call a normally hyberbolic subset $M$ attracting if all eigenvalues of $(D_xf)(p)$ have negative real parts for $p\in M$; similarly $M$ is called repelling if all eigenvalues have positive real parts. If $M$ is normally hyperbolic and neither attracting nor repelling we say it is of saddle-type. A typical class of points where normal hyperbolicity fails are \textit{fold points}. They are defined as the points where the critical manifold $C$ is locally parabolic with respect to the fast directions. In other words, at a fold point $p_*$ one requires that $f(p_*, 0) = 0$ and that $(D_xf)(p_*, 0)$ is of rank $m - 1$ with left and right null vectors $w$ and $v$, such that $w\cdot [(D_{xx}f)(p)(v, v)] \ne 0$ and $w \cdot [(D_yf)(p)] \ne 0$. 

\begin{ex}(Example \ref{ex:fss} continued)
The critical manifold $C=\{y=x^2\}$ splits into one repelling part, one attracting part and a fold point at the origin:
\benn
C=C_l\cup\{(0,0)\} \cup C_r
\eenn
where $C_l=C\cap \{x<0\}$ is repelling, $C_r=C\cap \{x>0\}$ is attracting and $(x,y)=(0,0)$ is a fold point which is easily verified since
\benn
D_xf=\frac{\partial f}{\partial x}=-2x\qquad \text{and} \qquad D_{xx}f=\frac{\partial^2 f}{\partial x^2}=-2. 
\eenn
\end{ex}

We continue with the general case. If $(D_xf)$ has maximal rank the implicit function theorem applied to $f(x,y,0)=0$ locally provides a function $h(y)=x$ so that $C$ can be expressed as a graph. Hence the slow subsystem \eqref{eq:basic_sf} can be more succinctly expressed as:
\be
\label{eq:basic_sf1}
\dot{y}=g(h(y),y,0)
\ee    
We shall also refer to the flow induced by \eqref{eq:basic_sf1} as slow flow. To relate the dynamics of the slow flow to the dynamics of the full system for $\epsilon>0$ the next theorem is of fundamental importance.

\begin{thm}[\textit{Fenichel's Theorem}, \cite{Fenichel4,Jones}]
\label{thm:fenichel1}
Suppose $M=M_0$ is a compact normally hyperbolic submanifold of the critical manifold $C$. Then for $\epsilon>0$ sufficiently small the following holds:
\begin{itemize}
\item[(F1)] There exists a locally invariant manifold $M_\epsilon$ diffeomorphic to $M_0$. \textit{Local invariance} means that $M_\epsilon$ can have boundaries through which trajectories enter or leave.  
\item[(F2)] $M_\epsilon$ has a distance $O(\epsilon)$ from $M_0$.
\item[(F3)] The flow on $M_\epsilon$ converges to the slow flow as $\epsilon \ra 0$.
\item[(F4)] $M_\epsilon$ is $C^r$-smooth for any $r<\I$ (as long as $f,g\in C^\I$).
\item[(F5)] $M_\epsilon$ is normally hyperbolic and has the same stability properties with respect to the fast variables as $M_0$ (attracting, repelling or saddle-type).
\item[(F6)] For fixed $\epsilon>0$, $M_\epsilon$ is usually not unique but all manifolds satisfying (F1)-(F5) lie at a Hausdorff distance $O(e^{K/\epsilon})$ from each other for some $K>0$, $K=O(1)$. 
\end{itemize} 
We call a manifold $M_\epsilon$ a \textit{slow manifold}. Note that all asymptotic notation refers to $\epsilon\ra0$. The same conclusions as for $M_0$ hold (locally) for its stable and unstable manifolds:
\benn
W^s(M_0)=\bigcup_{p\in M_0} W^s(p),\qquad  W^u(M_0)=\bigcup_{p\in M_0} W^u(p)
\eenn
where we view points $p\in M_0$ as equilibria of the fast subsystem.
\end{thm} 

Figure \ref{fig:fig1}(b) shows an typical scenario where Fenichel's Theorem applies; there we picked a compact submanifold $M_0\subset C_r$ and obtained an associated slow manifold. In addition to Fenichel's Theorem we can also find coordinate changes the simplify a fast-slow system considerably near a critical manifold.

\begin{thm}[\textit{Fenichel Normal Form}, \cite{Fenichel4,JonesKaperKopell}]
\label{thm:FNform}
Suppose the origin $0\in C$ is a normally hyberbolic point with $m_u$ unstable and $m_s$ stable fast directions; choose a sufficiently small compact normally hperbolic subset $M_0\subset C$ containing the origin. Then there exists a smooth invertible coordinate change $(x,y)\mapsto (a,b,v)\in\R^{m_u}\times \R^{m_s}\times \R^n$ in a neighbourhood of $0$ so that a fast-slow system \eqref{eq:basic2} can be written as
\bea
\label{eq:FNform1}
a'&=& \Lambda(a,b,v,\epsilon)a\nonumber,\\
b'&=& \Gamma(a,b,v,\epsilon)b,\\
v'&=& \epsilon(m(v,\epsilon)+H(a,b,v,\epsilon)ab),\nonumber
\eea
where $\Lambda$, $\Gamma$ are matrix-valued functions. $\Lambda$ has $m_u$ positive and $\Gamma$ has $m_u$ negative eigenvalues. $H$ is bilinear and given in coordinates by
\be
\label{eq:Hbilinear}
H_i(a,b,v,\epsilon)ab=\sum_{r=1}^{m_s}\sum_{s=1}^{m_u}H_{irs}a_r b_s.
\ee
\end{thm}

The situation is illustrated in Figure \ref{fig:fig2}. The manifold $M_0$ perturbs to a slow manifold $M_\epsilon$ by Fenichel's Theorem. Then this slow manifold is ``straightened'' together with its stable and unstable manifolds that become coordinate planes.\\

\begin{figure}[htbp]
	\centering 
\psfrag{unstablefib}{unstable directions}
\psfrag{stablefib}{stable directions}
\psfrag{a1=0}{$\{a=0\}=W^s(M_\epsilon)$}
\psfrag{a2=0}{$\{b=0\}=W^u(M_\epsilon)$}
\psfrag{Meps}{$M_\epsilon$}
\psfrag{y}{$v$}
\psfrag{a2}{$b$}
\psfrag{a1}{$a$}                       
		\includegraphics[scale=0.8]{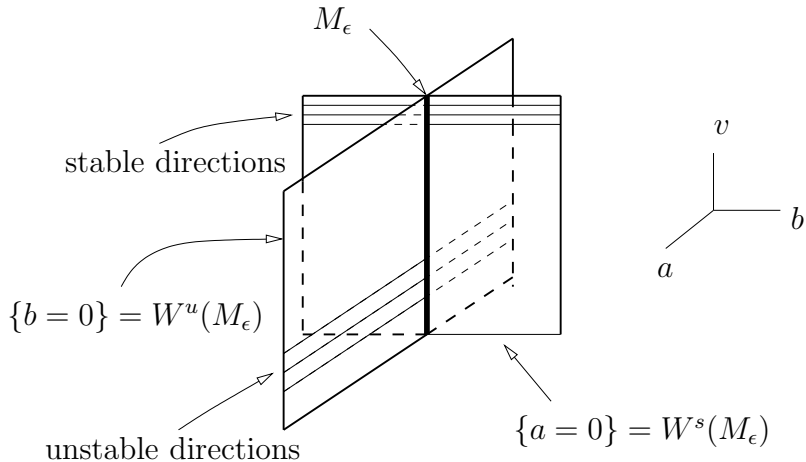}
	\caption{Illustration of Theorem \ref{thm:FNform}.}
	\label{fig:fig2}
\end{figure}

As a major overall conlusion of Fenichel Theory we get that the flow near a normally hyperbolic critical manifold is ``completely determined'' by the singular limit systems for $\epsilon=0$. We will show in Section \ref{sec:eepts} how this result reappears for the Conley index theory of fast-slow systems.

\section{Conley Index Theory}

In this section we describe the basic constructions from Conley Index Theory and how these have been adapted to fast-slow systems. When Conley \cite{ConleyEaston} studied these techniques that now bear his name it seems possible \cite{ConleyFastSlow1} that he also had applications to singular perturbation problems in mind. The idea of the theory is to convert a dynamical problem (e.g. ``Does this dynamical system have a heteroclinic orbit'') into an algebraic problem (e.g. ``What is the structure of a matrix?''). Several successful applications exist; see \cite{MischaikowMrozek} for a recent survey. We are going to outline only the basic techniques of the theory focusing on isolating neighbourhoods for fast-slow ODEs. In this context, the current theory can be found in \cite{Conley,MischaikowMrozekReineck,ConleyFastSlow1,ConleyFastSlow2,KokubuMischaikowOka}. For a more detailed introduction in the case of general finite-dimensional dynamical systems we refer to \cite{MischaikowMrozek}; the infinite-dimensional case is considered in \cite{Rybakowski}.\\

Let $\phi:\R\times \R^k \ra \R^k$ be a flow with $\phi=\phi(t,z)$. A compact set $N\subset \R^k$ is called an \textit{isolating neighbourhood} if 
\benn
\text{Inv}(N,\phi):=\{z\in\R^k|\phi(\R,z)\subset N\}\subset \text{int}(N)
\eenn  
where $\text{int}(N)$ denotes the interior of $N$. If we set $S:=\text{Inv}(N,\phi)$ then $S$ is called an \textit{isolated invariant set}. The situation is illustrated in Figure \ref{fig:fig3} where $\text{Inv}(N,\phi)$ is an unstable node.  

\begin{figure}[htbp]
\psfrag{N}{$N$}
\psfrag{L}{$L$}
\psfrag{q}{$q$}
	\centering
		\includegraphics[width=0.25\textwidth]{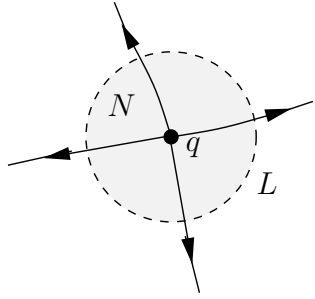}
	\caption{\label{fig:fig3}An isolating neighbourhood $N$ (shaded disk) of an unstable node $q$ is shown. The node $q$ is the invariant set of of $N$ i.e. $\text{Inv}(N,\phi)=\{q\}$. The boundary of $N$ (dashed circle) is denoted by $L$ and $(N,L)$ is easily seen to form an index pair.}	
\end{figure}

\begin{defn}
\label{defn:indexpair}
Let $S$ be an isolated invariant set. A pair of compact sets $(N,L)$ with $L\subset N$ is called an \textit{index pair} for $S$ if the following conditions hold:
\begin{enumerate}
 \item[(a)] $S=\text{Inv}(\text{cl}(N-L))$ and $N-L$ is a neighbourhood of $S$.
 \item[(b)] $L$ is positively invariant in $N$ i.e. for any $z\in L$ and $\phi([0,t],z)\subset N$ then $\phi([0,t],z)\subset L$.
 \item[(c)] $L$ is an \textit{exit set} for $N$ i.e. for any $z\in N$ and $t_1>0$ such that $\phi(t_1,z)\not\in N$ then there exists $t_0\in[0,t_1]$ for which $\phi([0,t_0],z)\subset N$ and $\phi(t_0,z)\in L$. 
\end{enumerate}
\end{defn} 

We define the \textit{Conley index} of $S$ as 
\benn
CH_*(S):=H_*(N,L)
\eenn
where $H_*$ is relative homology \cite{Hatcher,Spanier}. Note that an alternative way to define the Conley index would be to consider cohomology and set $CH^*(S):=H^*(N,L)$. The strategy to use the Conley index for dynamical systems usually proceeds along the following lines:
\begin{enumerate}
 \item[(S1)] Find an isolating neighbourhood $N$.
 \item[(S2)] Determine an index pair $(N,L)$.
 \item[(S3)] Calculate the Conley index.
 \item[(S4)] Use the calculation to prove a result about $\text{Inv}(N,\phi)$.
\end{enumerate}  
We will focus on (S1)-(S2) in the context of fast-slow systems. The main question is whether we can use the fast-slow structure to find an index pair. Re-writing a general fast-slow system \eqref{eq:basic2} on the fast time scale with $z=(x,y)$ will be convient
\be
\label{eq:sipsumform}
z'=F_0(z)+\sum_{i=1}^j \epsilon^i F_i(z)+o(\epsilon^j).
\ee
We denote the flow of \eqref{eq:sipsumform} by $\phi_\epsilon:\R\times \R^{m+n}\ra \R^{m+n}$. Observe that $\phi_0$ is the flow of the fast subsystem \eqref{eq:basic_fss}. It is easy to see that if $N$ is an isolating neighbourhood for $\phi_0$ then it is also an isolating neighbourhood for $\phi_\epsilon$. The problem is that usually $N$ will not be an isolating neighbourhood for $\epsilon=0$ but it still can be an isolating neighbourhood $\epsilon>0$. 

\begin{defn}
A compact set $N\subset \R^{m+n}$ is called a \textit{singular isolating neighbourhood} if $N$ is not an isolating neighbourhood for $\phi_0$ but there exists $\bar{\epsilon}$ such that $N$ is an isolating neighbourhood for $\phi_\epsilon$ with $\epsilon\in (0,\bar{\epsilon}]$.
\end{defn}

The next example illustrates, without proof, a singular isolating neighbourhood in a fast-slow system.

\begin{ex}
\label{ex:VdPb}
A time reversed version of Van der Pol's \cite{vanderPol,vanderPol1} equation is
\be
\label{eq:VdPsip_b}
\begin{array}{lcl}
x'&=&x^3/3-x-y,\\
y'&=&\epsilon x.\\
\end{array}
\ee
We shall consider the Van der Pol equation in more detail in Section \ref{sec:VdP}. For now it is useful to look ahead to Figure \ref{fig:VdP} that shows the critical manifold $C_0$ of \eqref{eq:VdPsip_b} and an orbit for $\epsilon=0$ composed of fast and slow subsystem trajectories. To prove that this orbit perturbs we want to construct a singular isolating neighbourhood; the dashed lines in Figure \ref{fig:VdP} indicate a possible guess for a such a neighbourhood $N$. In Section \ref{sec:VdP} we are going to prove that $N$ is a singular isolating neighbourhood. For now it is important to observe that $N$ is not an isolating neighbourhood for $\epsilon=0$ on the fast time scale (check it!). 
\end{ex}

To check whether a compact set is a singular isolating neighbourhood we temporalily decide to define the complications away.

\begin{defn}
\label{defn:sept}
Let $N$ be a compact set and let $z\in \text{Inv}(N,\phi_0)=:S$. We say that $z$ is a \textit{slow exit (entrance) point} if there is a neighbourhood $U$ of $z$ and an $\bar{\epsilon}>0$ such that for all $\epsilon\in (0,\bar{\epsilon}]$ there is a time $T(\epsilon,U)>0$ ($T(\epsilon,U)<0$) such that 
\benn
\phi_\epsilon(T(\epsilon,U),U)\cap U=\emptyset.
\eenn  
Let $S^-$ $(S^+)$ denote the set of slow exit (entrance) points. Furthermore define the following sets
\benn
S_\partial :=S\cap \partial N \qquad \text{and} \qquad S_\partial^\pm:=S_\partial \cap S^\pm.
\eenn
\end{defn}

To understand what slow exit and entrance points are and what characterizes them is the main goal of this paper and should be clear after Section \ref{sec:eepts}. Obviously Definition \ref{defn:sept} `cheats' by prescribing the dynamics of the slow motion under perturbation. Therefore the next result is very easy to prove.

\begin{thm}
\label{thm:sinbhd}
If $\text{Inv}(N,\phi_0)\cap \partial N$ consists of slow exit and entrance points then $N$ is a singular isolating neighbourhood i.e. it is an isolating neighbourhood for the full fast-slow system for sufficiently small $\epsilon>0$.  
\end{thm}

Hence we have reduced the problem to characterizing slow exit/entrance points in a more computable way. The next definition provides a technical notion which will be necessary for this task.

\begin{defn}
The \textit{average} of a function $h$ on $S\subset \R^{m+n}$, denoted $\text{Avg}(h,S)$, is the limit as $T\ra \I$ of the set of numbers
\benn
\left\{\left.\frac1T \int_0^T h(\phi_0(s,z))ds\right|z\in S\right\}.
\eenn 
We say that $h$ has \textit{strictly positive averages} on $S$ if $\text{Avg}(h,S)\subset (0,\I)$.
\end{defn}

In a seminal paper \cite{Conley} Conley was able to give computable conditions for slow exit and entry points. 

\begin{thm}(\cite{Conley})
\label{thm:csept}
A point $z\in S$ is a slow exit point if there exists a compact set $K_z\subset S$ invariant under $\phi_0$, a neighbourhood $U_z$ of the chain recurrent set $\mathcal{R}(K_z)$ of $K_z$, an $\bar{\epsilon}>0$ and a function $l:\text{cl}(U_z)\times [0,\bar{\epsilon}]\ra \R$ such that the following conditions are satisfied:
\begin{enumerate}
 \item[(a)] $\omega(z,\phi_0)=\bigcap_{t\in \R}\text{cl}\left(\{\phi_0(s,z):s>t\}\right)\subset K_z$; here $\text{cl}(.)$ denotes closure.
 \item[(b)] $l$ is of the form $l(w,\epsilon)=l_0(w)+\epsilon l_1(w)+\ldots+\epsilon^j l_j(w)$.
 \item[(c)] If $L_0=\{w|l_0(w)=0\}$ then $K_z\cap \text{cl}(U_z)=S\cap L_0\cap \text{cl}(U_z)$ and furthermore $l_0|_{S\cap \text{cl}(U_z)}\leq 0$.
 \item[(d)] Let $G_j(w)=\nabla_z l_0(w)\cdot F_j(w)+\nabla_z l_1(w)\cdot F_{j-1}(w)+\ldots+\nabla_z l_j(w)\cdot F_0(w)$. Then for some $k$, $G_k=0$ if $k<j$ and $G_j$ has strictly positive averages on $\mathcal{R}(K_z)$.
\end{enumerate}
A point is a slow entrance point if the same conditions hold under reversal of time. Points that satisfy (a)-(d) (satisfy the conditions under time reversal) are called \textit{C-slow exit (entrance) points}. The compact set $K_z$ is called a \textit{slow exit guide}.
\end{thm}

We note that the function $l$ should be viewed as a Lyapunov-type function for the dynamics near the slow exit/entrance point. In \cite{MischaikowMrozekReineck} the authors claim that ``the only dynamics which plays a role in the calculations (for Theorem \ref{thm:csept}) is that of $\phi_0$''. Formally this is not the case since higher-order terms $F_j$ for $j>0$ do enter crucially in (d). But as we shall see in the next section, the idea was intuitively correct. In fact, one should state that only the fast and slow flows of the singular subsystems play a role in the calculations. \\

Using Theorem \ref{thm:csept} we can often identify an isolating neighbourhood for a fast-slow system. Mischaikow, Mrozek and Reineck \cite{MischaikowMrozekReineck} give an analogous construction for index pairs. 

\begin{defn}
A pair of compact sets $(N,L)$ with $L\subset N$ is called a \textit{singular index pair} if $\text{cl}(N-L)$ is a singular isolating neighbourhood and there exists an $\bar{\epsilon}>0$ such that for all $\epsilon\in (0,\bar{\epsilon}]$
\benn
H_*(N,L)=CH_*(\text{Inv}(\text{cl}(N-L),\phi_\epsilon)).
\eenn
\end{defn}

The singular index pair should be characterized by similar conditions as the usual index pair described in Definition \ref{defn:indexpair}. From the exit set requirement we know that $L$ has to contain the \textit{immediate exit set} of $N$
\benn
N^-:=\{z\in \partial N|\phi_0((0,t),z)\not\subset N \text{ for all }t>0\}.
\eenn
Regarding positive invariance, it turns out that give $Y\subset N$ one has to consider the \textit{pushforward set} in $N$ under the flow $Y$ defined by
\benn
\rho(Y,N,\phi_0):=\{z\in N|\exists w\in Y,t\geq 0 \text{ s.t. }\phi_0([0,t],w)\subset N,\phi_0(t,w)=z\}.
\eenn
Basically $\rho(Y,N,\phi_0)$ consists of points in $N$ that can be reached from $Y$ by a positive trajectory in $N$; observe that by construction we must have $Y\subset \rho(Y,N,\phi_0)$. In addition, we also must consider a special version of the unstable manifold of a point lying in $N$
\benn
W^u_N(Y):=\{z\in N|\phi_0((-\I,0),z)\subset N \text{ and }\alpha(z,\phi_0)\subset Y\}.
\eenn
Again we observe that $Y\subset W^u_N(Y)$. Before we can state the theorem about characterizing singular index pairs, one last definition is needed. 

\begin{defn}
A slow entrance point $z$ is called a \textit{strict slow entrance point} if there exists a neighbourhood $V$ of $z$ and an $\bar{\epsilon}>0$ such that if $v\in V\cap N$ and $\epsilon\in(0,\bar{\epsilon}]$ then there exists a time $t_v(\epsilon)$ such that
\benn
\phi_\epsilon([0,t_v(\epsilon)],v)\subset N.
\eenn
The set of strict slow entrance points will be denoted by $S^{++}_\partial$.
\end{defn}

\begin{thm}(\cite{MischaikowMrozekReineck})
\label{thm:sipmis}
Let $N$ be a singular isolating neighbourhood. Assume
\begin{enumerate}
 \item[(A)] $S_\partial^-$ consists of C-slow exit points.
 \item[(B)] $S_\partial \subset S_\partial^{++}\cup S_\partial^-$.
 \item[(C)] $(S_\partial^{++}-S_\partial^-)\cap \text{cl}(N^-)=\emptyset$.
\end{enumerate}
 For each $z\in S_\partial^-$, let $K_z$ be a slow exit guide for $z$. Define
 \benn
 L:=\rho(\text{cl}(N^-),N,\phi_0)\cup W^u_N\left(\bigcup_{z\in S^-_\partial}\mathcal{R}(K_z)\right).
 \eenn
 If $L$ is closed then (N,L) is a singular index pair.
\end{thm}

Observe that for Theorem \ref{thm:csept} and Theorem \ref{thm:sipmis} it is crucial to determine which points are slow exit/entrance points. The conditions (a)-(d) given in Theorem \ref{thm:csept} are complicated. The goal of this paper is to simplify these conditions. 
 
\section{Equilibrium Exit Points}
\label{sec:eepts}

Let $N$ be a compact set and let $z_0=(x_0,y_0)\in \text{Inv}(N,\phi_0)=:S$ where $\phi_0$ denotes the flow of the fast subsystem. Let $C$ denote the critical manifold. We make the following assumptions:

\begin{enumerate}
 \item [(A1)] $z_0\in C\cap \partial N$.
 \item [(A2)] $C$ is a normally hyperbolic manifold at $z_0$ and locally given as a graph $x=h(y)$.
 \item [(A3)] $\partial N$ is smooth and parallel to the fast fibers near $z_0$. 
 \item [(A4)] The slow flow $\dot{y}=g(h(y),y)$ is transverse to $\partial N$ near $z_0$. Let $\vec{n}$ denote the outward unit normal to $N$ at $z_0$; there are two cases: 
  \begin{enumerate}
   \item [(A4.1)] $\vec{n}\cdot (0,g(h(y_0),y_0))>0$, slow flow directed outward near $z_0$.
   \item [(A4.2)] $\vec{n}\cdot (0,g(h(y_0),y_0))<0$, slow flow directed inward near $z_0$. 
  \end{enumerate}
\end{enumerate} 

\begin{thm}
\label{thm:eepts}
Under conditions (A1)-(A4) the point $z_0$ is a slow exit/entrance point. If (A4.1) holds we have a slow exit point and for (A4.2) we get a slow entrance point.
\end{thm}

\textit{Remark:} Theorem \ref{thm:eepts} only requires knowledge about the fast and slow subsystems for $\epsilon=0$. This justifies more clearly than Theorem \ref{thm:csept} that a special Conley index theory for fast-slow systems is possible.

\begin{proof}
Suppose without loss of generality that (A4.1) holds so that we are trying to show that $z_0\in C$ is a slow exit point. We also work in a sufficiently small neighbourhood of $z_0$ for the rest of the proof. The goal is to verify the conditions (a)-(d) of Theorem \ref{thm:csept}. Using (A2) we apply Fenichel's Normal Form Theorem \ref{thm:FNform} to re-write the fast-slow system near $0$ as
\be
\label{eq:FNformPf1}
\begin{array}{lcl}
x'&=& \Omega(x,y,\epsilon)x,\\
y'&=& \epsilon(m(y,\epsilon)+H(x,y,\epsilon)x),\\
\end{array}
\ee  
where $H$ is bilinear as described in \eqref{eq:Hbilinear} and we have re-written the Fenichel coordinates as $x=(a,b)\in\R^m$ and $v=y\in\R^n$ with $\Omega=(\Lambda,\Gamma)$. Due to (A3), we obtain that $\partial N$ is locally given by $\{y=0\}$ and also locally we have $N=\{y_i\leq 0\text{ for all $i=1,\ldots,n$}\}$; see Figure \ref{fig:fig4}.\\ 

\begin{figure}[htbp]
\psfrag{N}{$N$}
\psfrag{partialN}{$\partial N$}
\psfrag{x1}{$x_1$}
\psfrag{x2}{$x_2$}
\psfrag{y}{$y$}
	\centering
		\includegraphics[width=0.65\textwidth]{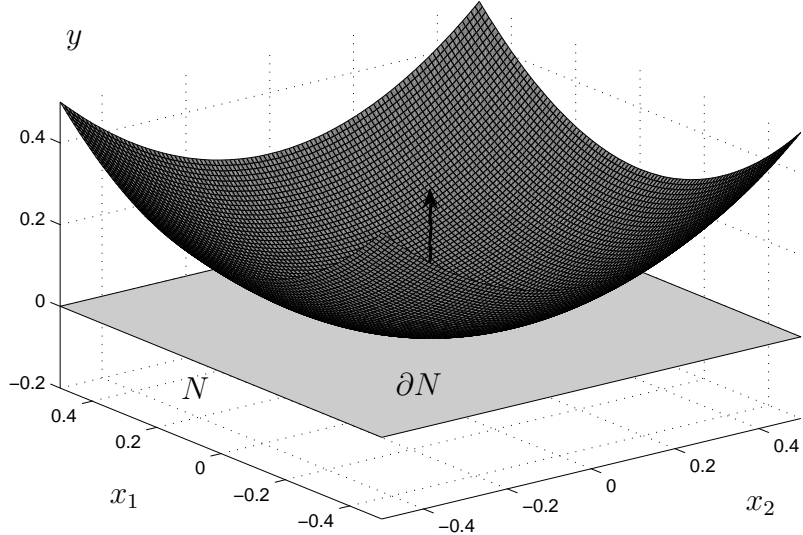}
	\caption{\label{fig:fig4}Illustration of the situation near a slow exit point at the origin in $\R^3$. The compact set $N$ is locally given by $\{y\leq 0\}$. The outer normal vector $e_1$ to $\partial N$ is also shown. The slow flow will point along this normal vector. The parabolic surface $L_0$ is the zero set $l(x,y)=0=l_0(x,y)$ .}	
\end{figure}

Note that (A4.1) implies that near $0$ we can rectify the slow flow so that \eqref{eq:FNformPf1} becomes:
\be
\label{eq:FNformPf2}
\begin{array}{lcl}
x'&=& \Omega(x,y,\epsilon)x,\\
y'&=& \epsilon(e_1+H(x,y,\epsilon)x),\\
\end{array}
\ee 
where $e_1=(1,0,\ldots,0)\in\R^n$. As a slow exit guide set $K=K_0:=\{0\}$ and observe that since $0\in C$ it is an equilibrium point for the fast subsystem. Therefore $\omega(0,\phi_0)=\{0\}$ and so $\omega(0,\phi)=K$ which verifies (a). Define the function $l$ by
\benn
l(x,y)=y_1-\sum_{j=1}^m (x_j)^2=l_0(x,y).
\eenn 
Notice that $l=l_0$ and so we find that (see also Figure \ref{fig:fig4}) 
\benn
L_0=\{(x,y)|l_0(x,y)=0\}=\left\{y_1=\sum_{j=1}^m (x_j)^2\right\}.
\eenn
Let $U$ be a sufficiently small neighbourhood around $0$ then $U\cap K=\{0\}$. We also have locally $S\cap \text{cl}(U)=N\cap C$ and therefore
\benn
S\cap \text{cl}(U)\cap L_0=\{0\}
\eenn
since $0= \sum_{j=1}^m (x_j)^2$ holds if and only if $x_j=0$ for all $j=1,2,\ldots, m$. Obviously $l_0|_{S\cap \text{cl}(U)}\leq 0$ and so (b)-(c) hold. For the last step, observe that $\mathcal{R}(0)=\{0\}$ and hence we have to verify that condition (d) holds at the origin i.e. 
\benn
G_j(z)=\nabla_z l_0(z)\cdot F_j(z)+\nabla_z l_1(z)\cdot F_{j-1}(z)+\ldots+\nabla_z l_j(z)\cdot F_0(z)
\eenn 
satisfies that for some $k$, $G_k(0)=0$ if $k<j$ and $G_j(0)>0$. We compute the gradient of $l_0$ 
\benn
\nabla_z l_0(x,y)=(-2x_1,-2x_2,\ldots,-2x_m,1,0,\ldots,0).
\eenn 
Since $\phi_0$ describes the fast flow, the first term $F_0$ in \eqref{eq:sipsumform} for the normal form \eqref{eq:FNformPf2} is given by
\benn
F_0(x,y)=(\Omega(x,y,0)x,0)^T.
\eenn 
This gives that $\nabla_z l_0(0)\cdot F_0(0)=0\cdot \Omega(x,y,0)x + e_1\cdot 0=0$. Hence $G_0(0)$ is identically zero. Next, we show that $G_1(0)$ is positive. We have
\benn
G_1(z)=\nabla_z l_0(z) \cdot F_1(z)=(-2x_1,-2x_2,\ldots,-2x_m,1,0,\ldots,0)\cdot(0, e_1+H(x,y,\epsilon)x)^T.
\eenn
Since $H(0,0,\epsilon)0=0$ we immediately get $G_1(0)=e_1\cdot (e_1)^T=1>0$ verifying (d). Therefore the original point $z_0$ is a slow exit point.
\end{proof}

The only condition that does not seem not quite natural for Theorem \ref{thm:eepts} is (A3). To illustrate that it is necessary consider the following example.

\begin{ex}
\label{ex:repelling}
Consider a fast-slow system with $(x,y)\in\R^2$ given by
\beann
x'&=&x,\\
y'&=&\epsilon.
\eeann
The solution is given by $(x(t),y(t))=(x(0)e^t,y(0)+\epsilon t)$. Fix some $m>0$ and let $N=\{(x,y)\in\R^2|y\leq mx\}$ locally near $0$ i.e. we truncate $N$ outside a suitable neighbourhood to make it compact. Now the origin is not a slow exit point although (A1)-(A2) and (A4.1) hold. Indeed, pick a neighbourhood $U$ of $0$ then there is $(x(0),y(0))\in U$ such that $x(0),y(0)>0$. For $\epsilon>0$ sufficiently small we can easily assure that $y(0)+\epsilon t<mx(0)e^t$ for all $t>0$ such that the trajectory starting at $(x(0),y(0))$ stays in $N$. 
\end{ex} 

Example \ref{ex:repelling} also indicates that the problem should not occur for attracting critical/slow manifolds.

\begin{prop}
\label{prop:simple}
Suppose (A1),(A2) and (A4) hold. Furthermore assume that $\partial N$ is locally linear and has an angle of order $O(1)$ to $C$ at $z_0$ and that $C$ is attracting at $z_0$. Then $z_0$ is a slow exit/entrance point. If (A4.1) holds we have a slow exit point and for (A4.2) we get a slow entrance point.
\end{prop}

\begin{proof}
In this case the proof is much simpler and we do not need Conley's Theorem \ref{thm:csept}. Again we can restrict without loss of generality to the case (A4.1). We will work under the assumption for the rest of this proof that $\epsilon>0$ has been chosen so that Fenichel Theory applies. Applying Fenichel's Normal Form Theorem as in Theorem \ref{thm:eepts} gives
\be
\label{eq:FNformPf1a}
\begin{array}{lcl}
x'&=& \Omega(x,y,\epsilon)x,\\
y'&=& \epsilon(m(y,\epsilon)+H(x,y,\epsilon)x),\\
\end{array}
\ee  
where now $\Omega(0,0,0)$ has $m$ negative eigenvalues. Let $U$ be a small neighbourhood around the origin. By Fenichel's Theorem \ref{thm:fenichel1} (F1) there exists a slow manifold $C_\epsilon$. Let $\gamma$ be a trajectory with an initial condition in $U$. Since $\Omega(0,0,0)$ has $m$ negative eigenvalues Fenichel's Theorem (F5) shows that $C_\epsilon$ is attracting. Therefore $\gamma$ gets attracted exponentially to $C_\epsilon$ or lies in $C_\epsilon$. Observe that $\partial N$ is not tangent to $C_\epsilon$ as it has an $O(1)$ angle to $C_0$. By (A4.1) and Fenichel's Theorem (F3) we find that $\gamma$ must leave $N$ after a time $T_\gamma(\epsilon)$. Since $U$ is bounded we can take the maximum of all times over $\text{cl}(U)$ 
\benn
T(\epsilon,U):=\max_{\gamma(0)\in \text{cl}(U)}T_\gamma(\epsilon).
\eenn
This verifies Definition \ref{defn:sept}.    
\end{proof}

Unfortunately Proposition \ref{prop:simple} is rarely helpful. One reason is that often it is convenient to make the critical manifold repelling near slow exit points; see examples in Section \ref{sec:appl}. The main reason is that in many important cases critical manifolds of saddle-type appear \cite{GuckenheimerKuehn2}. In fact, one of the most well-known examples, the 3D FitzHugh-Nagumo equation, has two fast variables and one slow variable with a critical/slow manifold of saddle type \cite{GuckenheimerKuehn1,GuckenheimerKuehn3}.  

\section{Periodic Orbit Exit Points}

In this section we shall not aim for the most general results but show some characterizations of slow exit points in the case of periodic orbits. We restrict to the case of fast-slow systems in $\R^3$ with two fast variables i.e. $(x,y)\in\R^2\times \R$ and
\bea
\label{eq:basic3}
x_1'&=&f_1(x,y)\nonumber,\\
x_2'&=&f_2(x,y),\\
y'&=&\epsilon g(x,y).\nonumber
\eea
Let $\gamma_y(t)\in \R^2$ denote a periodic orbit for the fast subsystem with period $T_y$ so that
\benn
\gamma_y(0)=\gamma_y(T_y),\quad \gamma'_y(t)=f(\gamma_y(t),y).
\eenn
Let $N:=[-K,K]^2\times [-K,0]\subset \R^3$ for $K>0$ so that the following assumptions hold (see Figure \ref{fig:fig5}):

\begin{enumerate}
 \item [(B1)] There exists family of hyperbolic periodic orbits $\{\gamma_y\}$ for $y\in[-\delta_0,\delta_0]$ for some $\delta_0>0$ in the fast subsystem.
 \item [(B2)] $\{\gamma_y\}_{y\in[-\delta_0,0]}\subset N$ and $\gamma_0\subset\text{int}([-K,K]^2\times\{0\})$.
\end{enumerate}

\begin{figure}[htbp]
\psfrag{N}{$N$}
\psfrag{pN}{$\partial N$}
\psfrag{gamma}{$\gamma_0$}
\psfrag{l0}{$L_0$}
\psfrag{x2}{$x_2$}
\psfrag{x1}{$x_1$}
\psfrag{y}{$y$}
	\centering
		\includegraphics[width=0.65\textwidth]{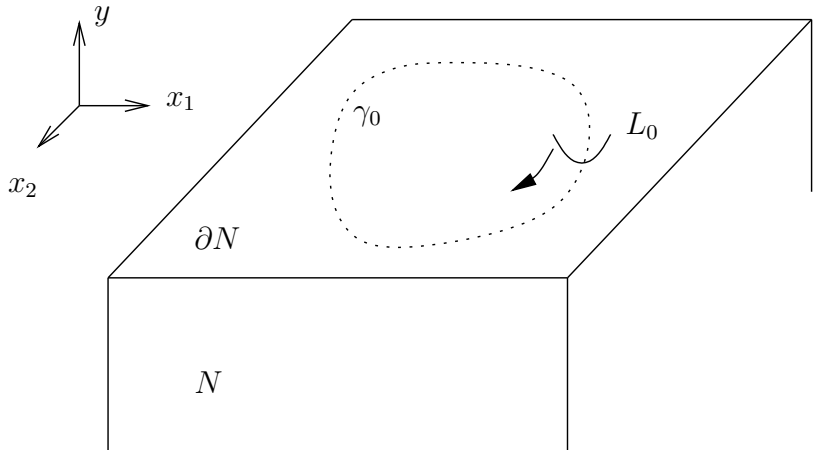}
	\caption{\label{fig:fig5}Sketch of the situation near periodic orbit $\gamma_0\subset \partial N$ in the fast subsystem. The parabolic surface $L_0$ given by the zero set $l_0(x,y)=0$ is defined by rotating the given parabola along $\gamma_0$.}	
\end{figure}

\begin{prop}
\label{prop:porbit}
Suppose (B1)-(B2) hold and assume that
\be
\label{eq:avg_cond}
\frac{1}{T_0} \int_0^{T_0} g(\gamma_0(s),0)ds>0
\ee
where $g$ is as given in \eqref{eq:basic3}. Then all points in $\gamma_0$ are slow exit points for $N$.
\end{prop}

\begin{proof}
The proof is very similar to the argument for Theorem \ref{thm:eepts}. Let $z_e=(x_e,0)\in \gamma_0$ be any point in the periodic orbit contained in $\partial N$. Observe that $\omega(x_e,\phi_0)=\gamma_0$ and let $K=\gamma_0=\mathcal{R}(K)$. Let $U$ be an annular neighbourhood of $K$ contained in $N$, for example we can set
\benn
U=\left\{(x,y)\in\R^3:\min_{z_\gamma}\|z_\gamma-(x,y)^T\|_2<\delta_1\text{ for $z_\gamma\in\gamma_y$ with $y\in[-\delta_2,\delta_2]$}\right\}
\eenn
for $\delta_1,\delta_2>0$ sufficiently small. let $\pi_0:\R^3\ra \R^3$ denote the orthogonal projection onto $\gamma_0$. Now define
\benn
l(z)=l(x,y)=l_0(x,y):=y-\sum_{j=1}^2(x_j-\pi_0(z)_j)^2
\eenn 
where the subscript $j$ indicates the $x_j$-coordinate of a point. In the notation of Theorem \ref{thm:csept} we easily check that (c) holds and we also observe that for (d) we have $G_0(z)\equiv 0$ since $\nabla_z l_0=(0,0,1)^T$ and $F_0=(f_1,f_2,0)$. We also find that $G_1=(0,0,g(z))\cdot (0,0,1)^T$ and on $K$ we indeed have
\benn
\frac{1}{T_0} \int_0^{T_0} g(\gamma_0(s),0)ds>0
\eenn
which verifies (d) and shows that $z_e$ is a slow exit point. Noting that $z_e$ was arbitrary on $\gamma_0$ finishes the proof.
\end{proof}

Note that Proposition \ref{prop:porbit} has the rather obvious interpretation that a point is a slow exit point for a periodic orbit of the fast subsystem if it lies on the boundary of the compact set $N$ and the average slow drift moves it outside of $N$. It is more interesting to re-interpret the condition \eqref{eq:avg_cond}.\\

We want to deal with families of periodic orbits in the fast subsystem. For normally hyperbolic parts of the critical manifold we know that there is a slow flow on a slow manifold that is $O(\epsilon)$-close. Next, we recall an analog of the result for periodic orbits of the fast subsystem. The idea is to find a flow that approximates the flow on the family of periodic orbits. Consider the fast system
\be
\label{eq:avg_fss}
\frac{dx}{dt}=x'=f(x,y)
\ee
such that \eqref{eq:avg_fss} has a continuous family of periodic orbits $\gamma_y(t)$ for each value of $y$ in some neighbourhood $D_0$ of $y=y_0$ with period $T_y$ that is uniformly bounded so that there are constants $T^a,T^b>0$ such that $T^a\leq T_y\leq T^b$. For simplicity we shall also assume that each orbit $\gamma_y(t)$ is asymptotically stable with respect to the fast variables. It seems plausible that the full fast-slow system should have solutions $(x(\tau),y(\tau))$ such that the fast motion is approximated by the family of rapid oscillating periodic orbits:
\benn
x(\tau)\approx \gamma_y\left(\frac{\tau}{\epsilon}\right).
\eenn   
Formally plugging this result into the slow equation yields
\benn
\dot{y}=g(\gamma_y(\tau/\epsilon),y).
\eenn
The idea is that the slow motion on the family of periodic orbits can be obtained by averaging out the fast oscillations. Hence we might consider
\be
\label{eq:avg_int1}
\dot{Y}=\bar{g}(Y):=\frac{1}{T_Y}\int_0^{T_Y}g(\gamma_Y(t),Y)dt.
\ee 
It will be convenient to make a change of variable $t=T_Y\theta$ and to set $\Gamma_Y(\theta)=\gamma_Y(T(Y)\theta)$. This transforms \eqref{eq:avg_int1} to
\be
\label{eq:avg_int2}
\dot{Y}=\int_0^1g(\Gamma_Y(\theta,Y))d\theta.
\ee 
Assume that the solution $Y(\tau)$ with initial condition $Y(0)=Y_0$ stays inside $D_0$ for $0\leq \tau\leq \tau_1$. Then a classical theorem shows that our averaging procedure really produces the correct result with an error of order $O(\epsilon)$.

\begin{thm}(\cite{PontryaginRodygin,BerglundGentz})
\label{thm:avg_fs}
Let $x_0$ be sufficiently close to $\Gamma_{Y_0}(\theta_0)$ for some $\theta_0$. Then there exists a function $\theta(\tau)$ that satisfies a differential equation of the form
\benn
\epsilon \dot{\theta}=\frac{1}{T_Y}+O(\epsilon).
\eenn
Furthermore the following estimates hold
\beann
x(\tau)&=&\Gamma_Y(\theta(\tau))+O(\epsilon),\\
y(\tau)&=&Y(\tau)+O(\epsilon),
\eeann
for $O(\epsilon|\log \epsilon|)\leq \tau\leq \tau_1$. 
\end{thm}

Therefore we observe that the averaged systems \eqref{eq:avg_int1}-\eqref{eq:avg_int2} appear in the condition \eqref{eq:avg_cond}. This means that points on the periodic orbit contained in $\partial N$ are slow exit or entry points if the averaged flow is transverse to $\partial N$; we should view this averaged flow as a ``slow flow'' on the family of fast periodic orbits. Hence we have provided an analog for the condition of slow exit and entry points on the critical manifold.\\

The next generalization step is now obvious. Consider a general family of invariant sets (e.g. tori) in the fast susbsytem. We can again average over the invariant measure of this family in the case of periodic orbits; the transversality conditions of this averaged flow will be exactly analogous to the previous cases. In practical applications this scenario does not seem to be needed very often as it does require three or more fast dimensions or a family of fast subsystems with one additional free parameter beyond the $y$-variables to be generic. 

\section{Examples}
\label{sec:appl}

\subsection{The Van der Pol Equation}
\label{sec:VdP}

We re-consider Example \ref{ex:VdPb}. The time reversed version of Van der Pol's \cite{vanderPol,vanderPol1} equation is
\be
\label{eq:VdPsip}
\begin{array}{lcl}
x'&=&x^3/3-x-y,\\
y'&=&\epsilon x.\\
\end{array}
\ee

\begin{figure}[htbp]
\psfrag{x}{$x$}
\psfrag{y}{$y$}
\psfrag{C}{$C_0$}
	\centering
		\includegraphics[width=0.8\textwidth]{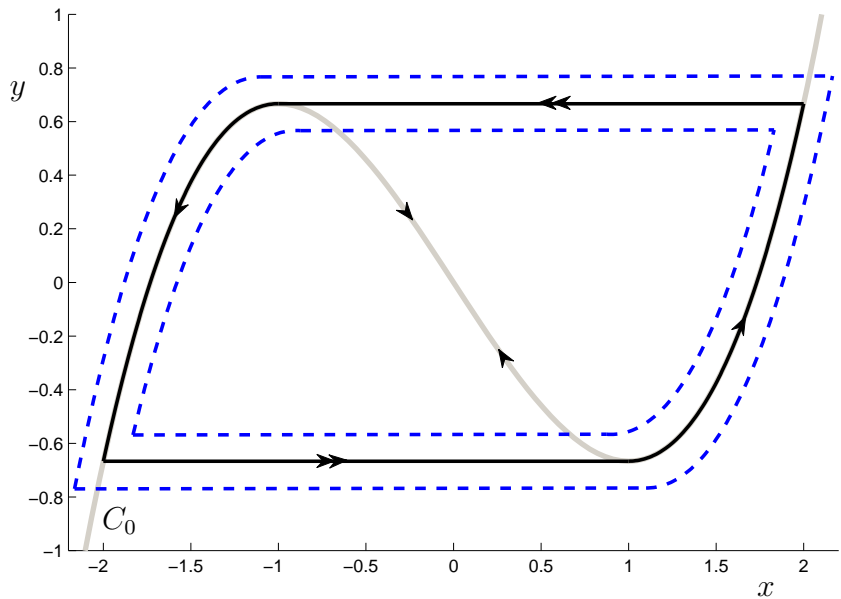}
	\caption{\label{fig:VdP}Critical manifold $C_0$ (grey), singular periodic orbit (black) and singular isolating neighbourhood (dashed blue) for Van der Pol's equation \eqref{eq:VdPsip}.}	
\end{figure}

The critical manifold $C_0$ of \eqref{eq:VdPsip} is given by 
\benn
C_0=\{(x,y)\in\R^2:y=x^3/3-x\}.
\eenn
Two fold points are located at $p_\pm=(\pm 1,\mp 2/3)$. They naturally split the critical manifold into three parts
\benn
C_l=C_0\cap \{x<-1\}, \quad C_m=C_0\cap \{-1\leq x \leq 1\}, \quad C_r=C_0\cap \{x>1\}.
\eenn
A singular periodic orbit for \eqref{eq:VdPsip} exists for $\epsilon=0$ consisting of concatenations of solutions of the fast and slow subsystems. We choose $N$ as a compact annulus containing the singular periodic orbit as indicated in Figure \ref{fig:VdP}. In this case, we have
\benn
S=\text{Inv}(N,\phi_0)=C_0\cap N.
\eenn

\begin{prop}
\label{prop:VdP}
$N$ is a singular isolating neighbourhood and $(N,\partial N)$ is a singular index pair.
\end{prop}

\begin{proof}
Note that $\partial N \cap S$ consists of four points, two on $C_m$ and one point each on $C_l$ and $C_r$. Since (A1)-(A4) hold, Theorem \eqref{thm:eepts} implies that all four points are slow exit points as the slow flow is transverse at each point and pointing outwards with respect to $N$. Therefore Theorem \ref{thm:eepts} implies that $N$ is a singular isolating neighbourhood. Next, we can apply Theorem \ref{thm:sipmis} to see that $L=\partial N$ and so $(N,L)$ is a singular index pair.
\end{proof} 

A direct Conley index calculation for $(N,L)$, in combination with the existence of a Poincar\'{e} section, can now be used to show that for $\epsilon>0$ sufficiently the singular periodic orbit perturbs to a periodic orbit of the full system \eqref{eq:VdPsip}; see \cite{MischaikowMrozek} for the detailed calculation.\\

\textit{Remark:} Proposition \ref{prop:VdP} is well-known in Conley index and serves as one of the basic examples how to apply the theory to show the existence of a non-trivial invariant set in a dynamical system. We emphasize here that with our characterization of slow exit points in Theorem \ref{thm:eepts}, the problem has been reduced to the minimum amount of work regarding the checking of theorems; our transversality condition of the slow flow is much easier to understand and check than the conditions of Conley's Theorem \ref{thm:csept}. 

\subsection{A Bursting Model}

We consider a modified Morris-Lecar model first proposed by Rinzel and Ermentrout \cite{RinzelErmentrout}:
\bea
\label{eq:terman}
x_1'&=&y-0.5(x_1+0.5)-2x_2(x_1+0.7)-0.5\left(1+\tanh\left(\frac{x_1+0.01}{0.15}\right)\right)(x_1-1),\nonumber\\
x_2'&=&1.15\left(0.5\left(1+\tanh\left(\frac{x_1-0.1}{0.145}\right)\right)-x_2\right)\cosh\left(\frac{x_1-0.1}{0.29}\right),\\
y'&=&\epsilon(k-x_1)\nonumber,
\eea
where $k$ is a parameter and $0\leq \epsilon\ll 1$. We note that \eqref{eq:terman} exhibits special periodic orbits that are examples of bursting oscillations; see \cite{Izhikevich} for more details. Terman \cite{Terman} and Guckenheimer and Kuehn \cite{GuckenheimerKuehn2} investigated \eqref{eq:terman} further focusing on the deformation of the periodic orbits under parameter variation related to a phenomenon called ``spike adding''. We shall not discuss these results further but refer to the original references. The important point in the current context is that the periodic orbits play a key role in the dynamics. Our goal is to construct a singular isolating neighbourhood for \eqref{eq:terman}. In Figure \ref{fig:bif_diag} we show a bifurcation diagram for the fast subsystem
\bea
\label{eq:terman_fss}
\begin{array}{lcl}
x_1'&=&y-0.5(x_1+0.5)-2x_2(x_1+0.7)-0.5\left(1+\tanh\left(\frac{x_1+0.01}{0.15}\right)\right)(x_1-1),\\
x_2'&=&1.15\left(0.5\left(1+\tanh\left(\frac{x_1-0.1}{0.145}\right)\right)-x_2\right)\cosh\left(\frac{x_1-0.1}{0.29}\right),\\
\end{array}
\eea
where we regard $y$ as a parameter. The diagram shows the continuation of an equilibrium point which traces out a projection of the critical manifold $C_0=\{x_1'=0=x_2'\}$. All continuation calculations have been carried out using MatCont \cite{MatCont}. The two fold points $p_{l,r}$ are fold (or saddle-node) bifurcations of the fast subsystem and at these points normal hyperbolicity is lost. They are located at
\beann
p_l&\approx &(-0.0337,-0.0207,0.1365)=:(x_{1,l},x_{2,l},y_{l}),\\
p_r& \approx &(-0.2449,0.0832,0.0085)=:(x_{1,r},x_{2,r},y_{r}).
\eeann

\begin{figure}[htbp]
\centering
\psfrag{H}{Hopf}
\psfrag{LP}{fold point}
\psfrag{NS}{neutral saddle}
\psfrag{x2}{$y$}
\psfrag{y}{$x_2$}
\includegraphics[width=0.8\textwidth]{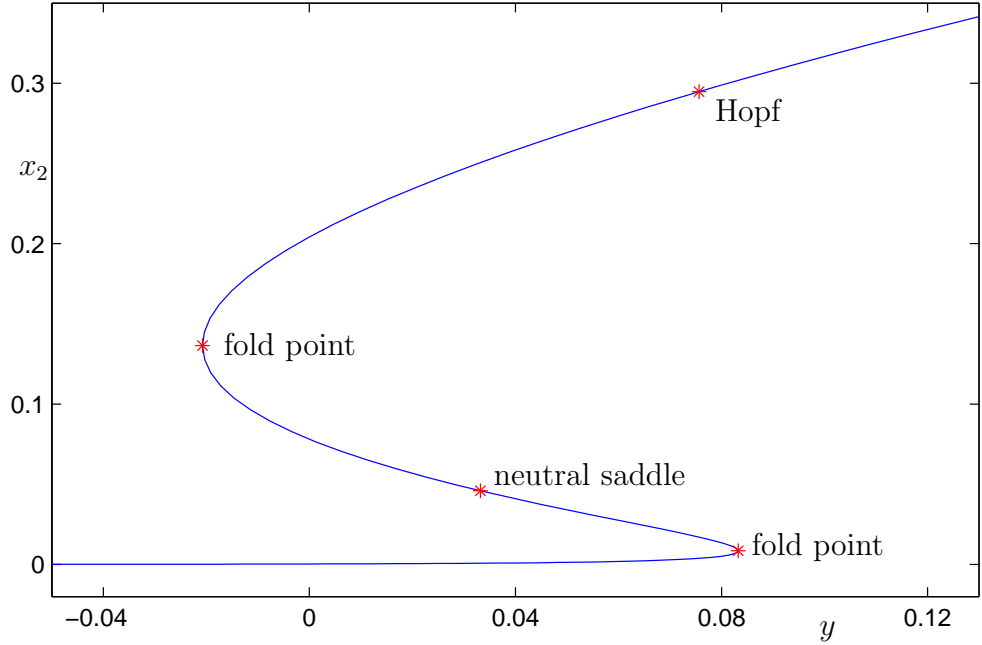}  
\caption{ \label{fig:bif_diag}Bifurcation diagram for the equilibria of \eqref{eq:terman_fss}. Note that the S-shaped curve of equilibria also represents the projection of the critical manifold into $(x_2,y)$-space.}
\end{figure}
  
The normally hyperbolic parts of $C_0$ are separated by the fold points into three branches
\benn
C_b=C_0\cap\{y<y_r\},\quad C_m=C_0\cap\{y_r< y< y_l\}, \quad C_u=C_0\cap\{y>y_l\}
\eenn
representing the lower, middle and upper parts of the S-shaped curve in Figure \ref{fig:bif_diag}. It is easy to check that $C_b$ is attracting and $C_m$ is of saddle-type. At $y_H\approx 0.075658$ we find a subcritical Hopf bifurcation of \eqref{eq:terman_fss}. The upper part of the critical manifold $C_u$ is repelling for $y<y_H$ and attracting for $y>y_H$. The unstable periodic orbits generated in the Hopf bifurcation of \eqref{eq:terman_fss} undergo further bifurcations as indicated in Figure \ref{fig:porbits}. The bursting periodic orbits for the full system contain the perturbation of a segment connecting $p_r$ to the stable periodic orbits in the fast subsystem at $y=y_r$.\\

\begin{figure}[htbp]
\centering
\includegraphics[width=0.8\textwidth]{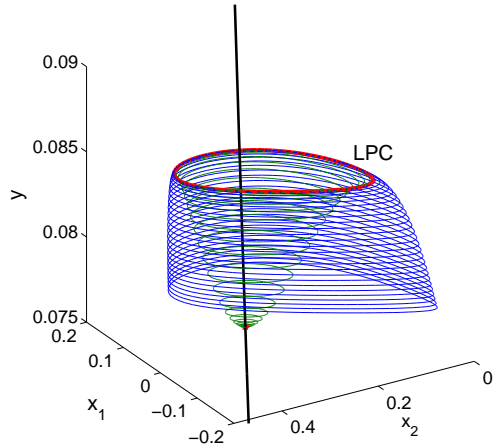}  
\caption{ \label{fig:porbits}Continuation of periodic orbits generated in the Hopf bifurcation for $y=y_H$. The smaller inner orbits (in green) are unstable and the larger outer orbits (in blue) are stable. Stability changes at a saddle-node of limit cycles (LPC, limit point of cycles, in red). The thick middle curve indicates part of $C_u$. Note that only a discrete number of periodic orbits are shown from the numerical calculation. }
\end{figure}

We shall not discuss how the periodic orbits connect back to $C_b$ (see \cite{GuckenheimerKuehn2}) but note that a compact set $N$ that is a potential singular isolating neighbourhood can have fast subsystem periodic orbits in $\partial N$. We shall focus on
\benn
N:=\{(x_1,x_2,y)\in\R^3:(x_1,x_2)\in K,-0.04\leq y\leq 0.084\}
\eenn 
for a suitably chosen compact rectangle $K\subset \R^2$ so that $N$ contains $p_l$ and $p_r$. Note that this choice of $N$ is only a first step in the analysis of bursting orbits and has to be refined to prove their existence; see the technical problems discussed in \cite{ConleyFastSlow1,ConleyFastSlow2}. The key point is that $\partial N\cap \{y=0.084\}$ contains two periodic orbits $\gamma_{1,2}$ for the fast subsystem. Figure \ref{fig:integrals}(a) shows the two periodic orbits.\\

\begin{figure}[htbp]
\centering
\psfrag{L1}{$\gamma_1$}
\psfrag{L2}{$\gamma_2$}
\psfrag{k}{$k$}
\psfrag{Int}{$I_j$}
\psfrag{x2}{$x_2$}
\psfrag{x1}{$x_1$}
\psfrag{(a)}{(a)}
\psfrag{(b)}{(b)}
\includegraphics[width=0.8\textwidth]{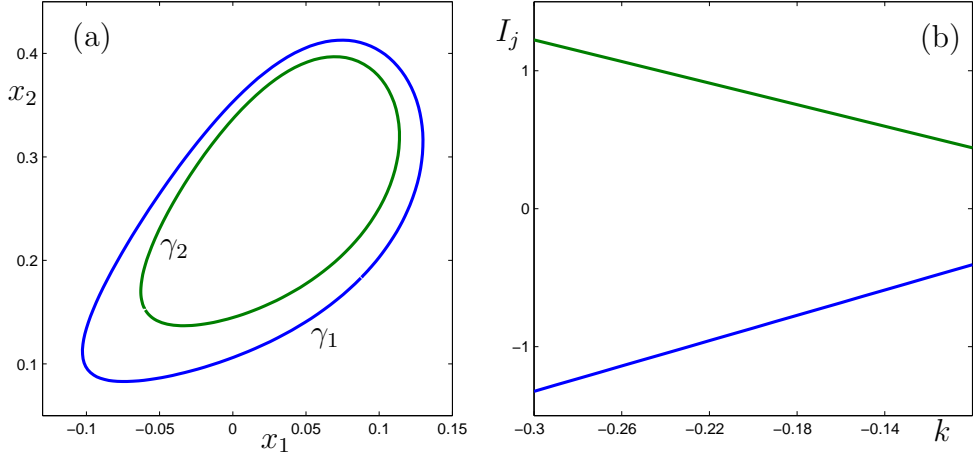}  
\caption{\label{fig:integrals} (a) Two periodic orbits in the fast subsystem at $y=0.084$. (b) Integrals given in \eqref{eq:int_terman} for the two periodic orbits depending on the parameter $k$.}
\end{figure}

From Proposition \ref{prop:porbit} we deduce that points on $\gamma_j$ for $j=1,2$ are slow exit/entry points depending on the sign of the integral
\be
\label{eq:int_terman}
I_j:=\int_0^{T_j} g(\gamma_j(s),0)ds=\int_0^{T_j} k-(\gamma_j(s))_1ds
\ee
where $T_j$ is the period of $\gamma_j$. Figure \ref{fig:integrals}(b) shows the value of $I_j$ for different values of $k\in[-0.3,-0.1]$ which are typical values within which bursting periodic orbits occur \cite{GuckenheimerKuehn2}. We find that in this parameter range the outer cycle $\gamma_1$ consists of slow entry points while the inner cycle $\gamma_2$ consists of slow exit points. Note that $\partial N\cap \{y=0.084\}$ also contains a point in $C_0$ at which the slow flow is transverse to $\partial N$. Furthermore we can choose the compact rectangle $K$ in the definition of $N$ so that there is only one more point in $\partial N\cap C_0$ lying on $C_b$. At this point the slow flow is again transverse to $\partial N$. By Theorem \ref{thm:eepts} both points are slow exit or entrance points. Next, we apply Theorem \ref{thm:sinbhd} to conclude that $N$ is a singular isolating neighbourhood in this case.\\

\textit{Remark:} Note that it is much more complicated to construct a singular index pair due to the presence of the Hopf bifurcation point of the fast subsystem on $C_0$. We postpone this question on how to modify $N$ to future work.  

\bibliographystyle{plain}
\bibliography{../my_refs}

\end{document}